\documentclass[10pt,article,reqno]{amsart}
\usepackage{amssymb}

\theoremstyle{plain}
\newtheorem{theorem}{Theorem}
\newtheorem{lem}{Lemma}

\newtheorem{cor}{Corollary}

\newcommand\ran{\operatorname{ran}}
\newcommand\Aff{\operatorname{Aff}}

\newcommand\Area{\operatorname{\blacklozenge}}

\newcommand{\N}{\mathbb{N}}

\newcommand{\R}{\mathbb{R}}

\begin{document}

\title[On $n$-norm preservers and the Aleksandrov problem]{On $n$-norm preservers and the Aleksandrov conservative $n$-distance problem}
\author{Gy.~P.~Geh\'er}
\thanks{The author was also supported by the "Lend\" ulet" Program (LP2012-46/2012) of the Hungarian Academy of Sciences and by the Hungarian National Research, Development and Innovation Office -- NKFIH (grant no.~K115383).}
\address{MTA-SZTE Analysis and Stochastics Research Group\\
Bolyai Institute\\
University of Szeged\\
H-6720 Szeged, Aradi v\'ertan\'uk tere 1, Hungary}
\address{MTA-DE "Lend\"ulet" Functional Analysis Research Group, Institute of Mathematics\\
University of Debrecen\\
H-4002 Debrecen, P.O. Box 400, Hungary}
\email{gehergy@math.u-szeged.hu or gehergyuri@gmail.com}
\urladdr{http://www.math.u-szeged.hu/$\sim$gehergy/}
\keywords{$n$-normed space, $n$-norm preserver, $n$-isometry, Aleksandrov problem.}
\subjclass[2010]{Primary: 51M25, 51K05; Secondary: 47B49, 46B04, 46B20.}

\begin{abstract}
	The goal of this paper is to point out that the results obtained in the recent papers \cite{2010_nn,2007_nn,2004_nn2,2004_nn} can be seriously strengthened in the sense that we can significantly relax the assumptions of the main results so that we still get the same conclusions.
	In order to do this first, we prove that for $n \geq 3$ any transformation which preserves the $n$-norm of any $n$ vectors is automatically plus-minus linear. 
	This will give a re-proof of the well-known Mazur--Ulam-type result that every $n$-isometry is automatically affine ($n \geq 2$) which was proven in several papers, e.g. in \cite{2009_nn}. 
	Second, following the work of Rassias and \v Semrl \cite{RaSe}, we provide the solution of a natural Aleksandrov-type problem in $n$-normed spaces, namely, we show that every surjective transformation which preserves the unit $n$-distance in both directions ($n\geq 2$) is automatically an $n$-isometry. 
\end{abstract}

\maketitle


\section{Introduction and statements of the main results}
Characterisations of Euclidean motions under mild hypothesis goes back to 1953, when Beckman and Quarles proved in \cite{BeQu} that an arbitrary transformation of $\R^d$ ($d\geq 2$) which preserves unit Euclidean distance (in one direction) is automatically a Euclidean motion (see \cite{Be,Ju} for alternative proofs). 
A similar conclusion does not hold in general, when we consider another norm on $\R^d$, namely, it is quite easy to construct a map on $\R^2$ which preserves unit $\ell^\infty$-distance but it is not an $\ell^\infty$-isometry.
The problem of characterising those at least two, but finite dimensional normed spaces $X$ which satisfy the property that any transformation $\phi\colon X\to X$ which preserves distance 1 (in one direction) is an isometry, was raised by Aleksandrov in \cite{Al}, hence it is usually called the \emph{Aleksandrov conservative distance problem}. 
Recently the author of this paper managed to show that the answer is affirmative for a large class of norms on $\R^2$ (see \cite{Ge}) which includes all strictly convex norms (see also \cite{Tyszka} for that case).
But in full generality the Aleksandrov problem is still open even in $\R^2$.  
However, some nice modified version of the problem was solved, see e.g.~\cite{BeBe,RaSe}. 

Turning back to Euclidean spaces, we may ask what those transformations $\phi\colon  E \to F$ are which satisfy the following property: 
\begin{equation}\label{nVOPP_eq}\tag{$n$VOPP}
\begin{gathered}
\Area_n(x_1-x_0,\dots, x_n-x_0) = 1 \quad \Longrightarrow \\ \Area_n(\phi(x_1)-\phi(x_0),\dots, \phi(x_n)-\phi(x_0)) = 1 \qquad
(x_0,\dots x_n\in E).
\end{gathered}
\end{equation}
Here $E$ and $F$ are real inner product spaces, $n\in\N$, $2\leq n\leq \min(\dim E, \dim F)$, and $\Area_n(x_1-x_0,\dots, x_n-x_0)$ denotes the usual $n$-dimensional volume (will be simply called $n$-volume from now on) of the parallelepiped $\big\{\sum_{j=1}^n t_j(x_j-x_0) \colon t_j\in[0,1]\big\}$, spanned by the vectors $x_1-x_0,\dots x_n-x_0$. 
Lester and Martin proved that 2VOPP maps are Euclidean motions if $n = 2 < \dim E = \dim F < \infty$; and equiaffine transformations, if $n = 2 = \dim E = \dim F$, \cite{Le2,Le}. 
(An equiaffine transformation is an affine map $\phi\colon E\to F$ such that the determinant of the matrix of the linear part of $\phi$, represented in some orthonormal bases of $E$ and $\ran \phi$, has determinant $\pm 1$). 
This is the so-called \emph{Lester--Martin theorem}.
A similar result for $n$VOPP maps is given in \cite{Be2}. 

The notion of $n$-volume and $n$VOPP maps can be generalised in the following natural way. 
Let $X$ be a real vector space with $\dim X \geq n$, and let us consider an $n$-variable function $\|\cdot,\dots, \cdot\|\colon X^n\to [0,\infty)$ which satisfies the following properties for every $x_1,\dots x_n, \widetilde x_1\in X$:
\begin{itemize}
\item[($n$N1)] $\|x_1,\dots, x_n\| = 0 \iff x_1,\dots x_n$ are linearly dependent,
\item[($n$N2)] $\|x_1,\dots, x_n\| = \|x_{\sigma(1)},\dots, x_{\sigma(n)}\|$ holds for any permutation $\sigma\in S_n$,
\item[($n$N3)] $\|\alpha\cdot x_1,x_2,\dots, x_n\| = |\alpha|\cdot\|x_1,x_2,\dots, x_n\|$ for every $\alpha\in\R$, and 
\item[($n$N4)] $\|x_1+\widetilde x_1,\dots, x_n\| \leq \|x_1,\dots, x_n\| + \|\widetilde x_1,\dots, x_n\|$.
\end{itemize}
Then we call $(X;\|\cdot,\dots,\cdot\|)$ a (real) \emph{$n$-normed space}. 
Note that the following property of $n$-norms is an easy consequence of ($n$N1) and ($n$N4):
\begin{itemize}
	\item[($n$N5)] if $y,x_2\dots x_n$ are linearly dependent, then $\|x_1+y,\dots, x_n\| = \|x_1,\dots, x_n\|$.
\end{itemize}
Obviously $\Area_n(\cdot,\dots,\cdot)$ is an $n$-norm if $X$ is an inner product space. 
The investigation of $n$-normed spaces started in the second half of the twentieth century (see e.g.~\cite{DiGaWh,Ga,Mi}), and it is a widely-investigated area even today. 

Throughout this paper, if we do not say otherwise, $X$ and $Y$ will always denote real $n$-normed spaces, and the $n$-norms on them will be denoted by the symbol $\|\cdot,\dots,\cdot\|$. 
A transformation $\phi\colon X\to Y$ is called an \emph{$n$-isometry} if it satisfies
\begin{equation}\label{nI_eq}\tag{$n$I}
\|x_1-x_0,\dots, x_n-x_0\| = \|\phi(x_1)-\phi(x_0),\dots, \phi(x_n)-\phi(x_0)\|
\end{equation}
for every $x_0,\dots x_n \in X$.
The fact that all $n$-isometries which fixes 0 are automatically linear is very-well known and it is usually referred to as the \emph{Mazur--Ulam theorem for $n$-normed spaces}. 
However, in this paper we will investigate a more general class of transformations which will be called \emph{$n$-norm preservers}. 
They are mappings $\phi\colon X\to Y$ satisfying the following condition:
\begin{equation}\label{nNP_eq}\tag{$n$NP}
\|x_1,\dots, x_n\| = \|\phi(x_1),\dots, \phi(x_n)\| \quad (x_1,\dots x_n \in X).
\end{equation}
We will prove in this paper that every $n$-norm preserver is automatically a plus-minus linear mapping and utilising that we will re-prove the aforementioned Mazur--Ulam-type theorem.

The Aleksandrov problem can be raised in $n$-normed spaces, as a possible generalisation of the Lester--Martin theorem. 
We say that the map $\phi\colon X\to Y$ has the $n$DOPP property (according to \cite{2004_nn2,RaSe}) if it fulfilles the following:
\begin{equation}\label{nDOPP_eq}\tag{$n$DOPP}
\begin{gathered}
\|x_1-x_0,\dots, x_n-x_0\| = 1 \quad \Longrightarrow \\ \|\phi(x_1)-\phi(x_0),\dots, \phi(x_n)-\phi(x_0)\| = 1 \quad (x_0,\dots x_n \in X).
\end{gathered}
\end{equation}
The \emph{Aleksandrov problem for $n$-normed spaces} is to characterise those finite dimensional $n$-normed spaces $X$ such that any $n$DOPP mapping $\phi\colon X\to X$ is an $n$-isometry. 

For a moment let us consider two (1-)normed spaces $X$ and $Y$. 
Usually we cannot expect from a general 1DOPP transformation $\phi\colon X\to Y$ to have a nice form, if $\dim X = \dim Y = \infty$ or $\dim X < \dim Y$ (see e.g.~\cite{BeQu,De}). 
Therefore in \cite{RaSe}, Rassias and \v Semrl considered surjections on normed spaces which fulfils the so-called (1SDOPP) property (defined below). 
They managed to show that such transformations are close to be isometries. 
In light of that, we will call the transformation $\phi\colon X\to Y$ an $n$SDOPP mapping, if it satisfies the following property ($n\in\N$):
\begin{equation}\label{nSDOPP_eq}\tag{$n$SDOPP}
\begin{gathered}
\|x_1-x_0,\dots, x_n-x_0\| = 1 \quad \iff \\ \|\phi(x_1)-\phi(x_0),\dots, \phi(x_n)-\phi(x_0)\| = 1 \quad (x_0,\dots x_n \in X).
\end{gathered}
\end{equation}

The goal of this paper is to contribute to the Aleksandrov problem in $n$-normed spaces, and as a byproduct, to significantly strengthen the results of \cite{2010_nn,2007_nn,2004_nn2,2004_nn}. 
Namely, we significantly relax the assumptions in those papers so that the conclusions still remain the same.

A map $\phi\colon X\to Y$ between real vector spaces is called \emph{plus-minus linear}, if there exists a map $\epsilon\colon X\to \{-1,1\}$ such that $\epsilon(\cdot)\phi(\cdot)$ is a linear transformation. 
Now, we state our first main result. 

\begin{theorem}\label{pmlin_nn_thm}
Let $X$ and $Y$ be two real $n$-normed spaces with $n\geq 3$, and $\phi\colon X\to Y$ be a (not necessarily surjective) transformation which satisfies \eqref{nNP_eq}. 
Then $\phi(\cdot)$ is plus-minus linear. 
\end{theorem}

The proof relies on the fundamental theorem of projective geometry and a previous result of the author ((i) of \cite[Theorem 1]{G6}). 
For $\Area_n(\cdot,\dots,\cdot)$ we will obtain a consequence, Corollary \ref{T:RW_n_cor}, that can be considered as an additional result to \cite[Theorem 1]{G6}. 

Our second main result considers $n$SDOPP surjections. 

\begin{theorem}\label{nSDOPP_thm}
Let $X$ and $Y$ be two real $n$-normed spaces with $n\geq 2$, and $\phi\colon X\to Y$ be a surjective $n$SDOPP transformation. 
Then $\phi$ is an injective, affine map, and therefore it is an $n$-isometry. 
\end{theorem}

In the proof we will apply the fundamental theorem of affine geometry, and as a consequence, we will obtain a Lester--Martin type theorem in Corollary \ref{nSDOPP_cor}.



\section{Proofs}

We begin with stating the fundamental theorem of projective geometry, in the version in which it will be needed. 
For a real vector space $X$, we denote the \emph{projectivised space} (i.e.~the set of all one-dimensional subspaces) by $P(X)$. 
The element of $P(X)$ generated by $0\neq x\in X$ will be denoted by $[x] := \R\cdot x$. 
In general, if $M\subset X$, then $[M]$ will denote the subspace generated by the set $M$. 
If $L\subseteq X$ is a two-dimensional subspace, then $[L]$ is called a \emph{projective line}. 
The following theorem is a special case of \cite[Theorem 3.1]{Fa} (namely, when $W=\{0\}$ and $K_1 = K_2 = \R$).

\begin{theorem}[The fundamental theorem of projective geometry]\label{proj_thm}
Let $X$ and $Y$ be two real vector spaces of dimensions at least three.
Let us consider an arbitrary (not necessarily surjective) transformation $g\colon P(X)\to P(Y)$ such that the following conditions are satisfied:
\begin{itemize}
\item[(i)] $\ran g$ is not contained in a projective line,
\item[(ii)] $0\neq c\in[a,b]$ $(a\neq 0 \neq b)$ implies $g([c])\subseteq \big[g([a]),g([b])\big]$.
\end{itemize}
Then there exists an injective linear transformation $A\colon X\to Y$ such that we have
\[
g([x]) = [Ax] \quad (0\neq x\in X).
\]
Moreover, $A$ is unique up to a non-zero scalar factor. 
\end{theorem}

In fact, a generalised version of this theorem is stated in \cite{Fa} for vector spaces over division rings, where the existence of a semi-linear map is shown. 
But, as the only endomorphism of $\R$ is the identity, here all semi-linear maps are linear. 

Before we prove Theorem \ref{pmlin_nn_thm} we provide the following lemma, which could be generalised for $n$-norms on $n$-dimensional spaces in a similar way, however, we only need this special case here.

\begin{lem}\label{seged1}
	In a two-dimensional space every 2-norm is a non-zero scalar multiple of $\Area_2(\cdot,\cdot)$.
\end{lem}

\begin{proof}
	The statement is clear from the following calculation:
	\begin{align*}
	\big\|\alpha_1 x_1 + \alpha_2 x_2 ,\beta_1 x_1 + \beta_2 x_2\big\| & = \left\|\left(\alpha_2-\frac{\alpha_1}{\beta_1}\beta_2\right) x_2,\beta_1 x_1 + \beta_2 x_2\right\| \\
	= \left|\alpha_2-\frac{\alpha_1}{\beta_1}\beta_2\right| \cdot \big\|x_2,\beta_1 x_1 + \beta_2 x_2\big\| & = \left|\alpha_2-\frac{\alpha_1}{\beta_1}\beta_2\right| \cdot \big\|x_2,\beta_1 x_1\big\| \\
	& = \left|\det\left(
	\begin{matrix}
	\alpha_1 & \alpha_2\\
	\beta_1 & \beta_2\\
	\end{matrix}
	\right)\right| \cdot \big\|x_1, x_2\big\|^{(2)}_{[x_1, x_2]}
	\end{align*}
	whenever $\beta_1\neq 0$. 
	If $\beta_1 = 0 \neq \beta_2$, by interchanging the role of $x_1$ and $x_2$, we obtain the same equation. 
	Finally, if $\beta_1 = \beta_2 = 0$, then the above equation is trivially fulfilled. 
\end{proof}

Now, we are in the position to prove our first main result.

\begin{proof}[Proof of Theorem \ref{pmlin_nn_thm}]
By \eqref{nNP_eq}, $\phi$ preserves linear independence of $n$ vectors in both directions. 
We observe that this is true for $k$ vectors $(2\leq k < n)$ as well.
Indeed, on the one hand, it is obvious that $\phi$ preserves linear independence of $k$ vectors $(2\leq k < n)$ in one direction. 
On the other hand, if $\phi(x_1),\dots \phi(x_k)$ are linearly independent $(2\leq k < n)$, then there are some vectors $\phi(x_{k+1}),\dots \phi(x_n)$ such that the system $\phi(x_1),\dots \phi(x_n)$ is still linearly independent, since otherwise $\ran\phi$ would be contained in a subspace of dimension less than $n$.
Thus $x_1,\dots x_k$ are linearly independent too.

We define the \emph{projectivisation} of $\phi$ in the following way:
\begin{equation}\label{proj-ed_eq}
P_\phi\colon P(X)\to P(Y), \quad P_\phi([x]) = [\phi(x)].
\end{equation}
By the above observations it is apparent that $P_\phi$ is well-defined and that it satisfies the conditions of Theorem \ref{proj_thm}. 
Therefore we obtain an injective linear transformation $A\colon X\to Y$ (which is unique up to a scalar multiple) that satisfies 
\begin{equation}\label{phi-A2_eq}
[\phi(x)] = [Ax] \quad (0\neq x\in X),
\end{equation}
and which brings us one step closer to conclude the plus-minus linearity of $\phi$.

Next, we consider two arbitrary linearly independent vectors $x_1,x_2\in X$.
We set some other vectors $x_3,\dots x_n\in X$ such that the system $x_1,\dots x_n$ is still linearly independent. 
By the above observations we have
$$
\phi(x)\in[\phi(x_1),\phi(x_2)] \;\iff\; x\in[x_1, x_2] \quad (x\in X).
$$
It is quite straightforward that 
\[
\big\|\cdot,\cdot\big\|^{(2)}_{[x_1, x_2]} \colon
[x_1, x_2] \to \R_+, \quad
\big\|z_1,z_2\big\|^{(2)}_{[x_1, x_2]} := \|z_1,z_2,x_3,\dots, x_n\|
\]
and 
\begin{align*}
\big\|\cdot,\cdot\big\|^{(2)}_{[\phi(x_1), \phi(x_2)]} \colon &
[\phi(x_1), \phi(x_2)] \to \R_+, \\
& \big\|u_1,u_2\big\|^{(2)}_{[\phi(x_1), \phi(x_2)]} := \|u_1, u_2,\phi(x_3),\dots, \phi(x_n)\|
\end{align*}
define 2-norms on the subspaces $[x_1, x_2]$ and $[\phi(x_1), \phi(x_2)]$, respectively. 
Observe that the restriction 
\[
\phi|_{[x_1, x_2]}\colon [x_1, x_2] \to [\phi(x_1), \phi(x_2)]
\] 
satisfies
\[
\big\|\phi(z_1), \phi(z_2)\big\|^{(2)}_{[\phi(x_1), \phi(x_2)]} 
= \big\|z_1, z_2\big\|^{(2)}_{[x_1, x_2]} \qquad (z_1, z_2\in[x_1, x_2]).
\]
Now, an easy application of Lemma \ref{seged1} and \cite[Theorem 1 (i)]{G6} gives a bijective linear map $A_{[x_1, x_2]}\colon [x_1,x_2] \to [\phi(x_1),\phi(x_2)]$ such that
\begin{equation}\label{phi-A_eq}
\phi(u) \in \{A_{[x_1, x_2]}u, -A_{[x_1, x_2]}u\} \quad (u\in [x_1,x_2]),
\end{equation}
and this holds for every two linearly independent vectors $x_1,x_2\in X$. 
Clearly, by \eqref{phi-A_eq} and \eqref{phi-A2_eq},
\begin{equation}\label{proj_eq}
[Au] = [A_{[x_1, x_2]}u] \quad (u\in[x_1, x_2]).
\end{equation}
holds for every two linearly independent vectors $x_1, x_2\in X$. 
We claim that there is a non-zero constant $c_{[x_1, x_2]}$ such that we have 
\begin{equation}\label{const_2dim_eq}
c_{[x_1, x_2]}\cdot A|_{[x_1, x_2]} = A_{[x_1, x_2]}.
\end{equation}
Indeed, by \eqref{proj_eq} we have $A_{[x_1, x_2]}x_j = \alpha_j A x_j$ $(j=1,2)$ with some constants $\alpha_j\in \R\setminus\{0\}$, whence we get 
\begin{align*}
[A(x_1+x_2)] & = [Ax_1 + Ax_2] = \left[\tfrac{1}{\alpha_1}A_{[x_1, x_2]}x_1 + \tfrac{1}{\alpha_2}A_{[x_1, x_2]}x_2\right] \\
& = \left[A_{[x_1, x_2]}\left(\tfrac{1}{\alpha_1}x_1 + \tfrac{1}{\alpha_2} x_2\right)\right] 
= \left[A\left(\tfrac{1}{\alpha_1}x_1 + \tfrac{1}{\alpha_2} x_2\right)\right].
\end{align*}
But $A$ is injective, therefore we get $\alpha_1 = \alpha_2$. 
If we set $c_{[x_1, x_2]} = \alpha_1$, then this constant satisfies \eqref{const_2dim_eq}.

Finally, let us consider two pieces of two-dimensional subspaces $F_1$ and $F_2$ of $X$. 
If $\{0\}\neq F_1\cap F_2\neq F_1$, then by \eqref{phi-A_eq} and \eqref{const_2dim_eq} we obtain $c_{F_2} \in \{c_{F_1},-c_{F_1}\}$. 
If $\{0\} = F_1\cap F_2$, then there exists a third two-dimensional subspace $F_3$ such that $\{0\}\neq F_j\cap F_3\neq F_3$ $(j=1,2)$ holds, and applying the previous case we get $c_{F_2} \in \{c_{F_1},-c_{F_1}\}$. 
Since $\tfrac{1}{c}\cdot A$ ($c\neq 0$) also fulfils \eqref{proj_eq}, we may suppose without loss of generality that $c_{F}\in \{-1,1\}$ holds for every two-dimensional subspace $F\subseteq X$, and thus by \eqref{phi-A_eq} and \eqref{const_2dim_eq} $\phi$ is plus-minus linear.
\end{proof}

The following corollary is a supplementary result to \cite[Theorem 1]{G6}, where we do not have to assume completeness of the spaces, nor bijectivity of $\phi$.

\begin{cor}\label{T:RW_n_cor}
Let $E$ and $F$ be real inner product spaces, $n\in \N, 3 \leq n \leq \dim E$, and $\phi\colon E\to F$ be a transformation which satisfies
\begin{equation}\label{nVPP_eq}\tag{$n$VPP}
\Area_n(x_1,\dots, x_n) = \Area_n(\phi(x_1),\dots, \phi(x_n)) \quad (x_1,\dots x_n \in X).
\end{equation}
Then we have the following conclusions:
\begin{itemize}
\item[(i)] If $\dim E = n$, then there exist a function $\epsilon\colon E\to\{-1,1\}$ and an equiaffine linear transformation $A\colon E\to F$ such that the following holds:
\[
\phi(x) = \epsilon(x)Ax \qquad (x \in E).
\]
\item[(ii)] If $n < \dim E$, then there exist a function $\epsilon\colon E\to\{-1,1\}$ and a linear (not necessarily surjective) isometry $R\colon E\to F$ such that
\[ 
\phi(x) = \epsilon(x)Rx \qquad (x \in E) 
\]
is satisfied.
\end{itemize}
\end{cor}

\begin{proof}
By Theorem \ref{pmlin_nn_thm}, we immediately obtain the existence of a function $\epsilon\colon E\to\{-1,1\}$ such that $\epsilon(\cdot)\phi(\cdot)$ is linear. 
Since $\phi(\cdot)$ fulfilles the conditions of our statement if and only if $\epsilon(\cdot)\phi(\cdot)$ does, there is no loss of generality if we assume that $\phi$ is linear.
If $\dim E = n$, then the statement is clear.

Now, we assume $2 < n < \dim E$, and show that $\phi$ is an isometry.
Let us consider the restriction $\psi := \phi|_H\colon H \to \phi(H)$ into an arbitrary $(n+1)$-dimensional subspace $H\subseteq E$, where clearly $\dim\phi(H) = n+1$. 
Let us observe that whenever $U\colon \phi(H)\to H$ is an arbitrary linear isometry, then $\psi(\cdot)$ satisfies \eqref{nVOPP_eq} if and only if $U(\psi(\cdot))\colon H\to H$ does. 
By the polar decomposition (or singular value decomposition), there exists a suitable $U$ such that $U(\psi(\cdot))$ has a diagonal matrix representation with positive diagonal elements, in some orthonormal base $h_1,\dots h_{n+1}$ of $H$. 
Thus we may assume without loss of generality that $\psi(h_j) = d_j h_j$ is satisfied with some $d_j>0$ $(j=1,\dots n+1)$. 
Since $\Area_n(h_1,\dots, h_n) = 1 = \Area_n(h_2,\dots, h_{n+1})$, we obtain
\[
d_1\cdot\dots\cdot d_n = \Area_n(d_1 h_1,\dots, d_n h_n) = 1 = \Area_n(h_2,\dots, h_{n+1}) = d_2\cdot\dots\cdot d_{n+1},
\]
which further implies $d_1 = d_{n+1}$. 
Similarly, we get $d_1 = d_2 = \dots = d_{n+1}$, but this is possible only in the case when $d_1 = \dots = d_{n+1} = 1$, i.e. when $\phi|_H$ is an isometry. 
Since $H$ was arbitrary, our map is indeed an isometry.
\end{proof}

Next, we state a strong version of the fundamental theorem of affine geometry below, which is a special case of \cite[Theorem 2.1]{HuSe}. We call a mapping $\eta \colon X\to Y$ between two real vector spaces a \emph{lineation} if it maps any three collinear points of $X$ into collinear points of $Y$. 
It is straightforward to show that if $\eta$ is injective, then it is a lineation if and only if we have 
\[
\phi(\Aff(a,b))\subseteq\Aff(\phi(a),\phi(b)) \quad (a,b\in X, a\neq b),
\]
where $\Aff(M)$ denotes the affine subspace generated by $M \subseteq X$. 

\begin{theorem}[The fundamental theorem of affine geometry, \cite{HuSe}]\label{aff_thm}
Let $\eta \colon \R^2\to\R^2$ be an injective lineation whose range is not contained in any affine line. 
Then $\eta$ is an injective affine transformation.
\end{theorem}

However, we want to use Theorem \ref{aff_thm} for arbitrary real vector spaces. 
This extension can be obtained quite straightforwardly, as presented below. 

\begin{cor}\label{aff_cor}
Let $X$ and $Y$ be two real, at least two-dimensional vector spaces, and suppose that $\eta\colon X\to Y$ is an injective lineation whose range is not contained in an affine line. 
Then $\eta$ is an injective affine transformation.
\end{cor}

\begin{proof}
Let $a$ and $b$ be two different points in $X$. 
Since $\ran \eta \subsetneq \Aff(\eta(a),\eta(b))$, there exists a vector $c\in X$ such that $\eta(a), \eta(b)$ and $\eta(c)$ are affine independent. 
Clearly, $a, b$ and $c$ must be also affine independent, moreover, we have $\eta(\Aff(a,b,c)) \subseteq \Aff(\eta(a),\eta(b),\eta(c))$.
We consider the restriction 
\[
\eta|_{\Aff(a,b,c)}\colon \Aff(a,b,c) \to \Aff(\eta(a),\eta(b),\eta(c)).
\]
By Theorem \ref{aff_thm}, the function $\eta|_{\Aff(a,b,c)}$ preserves all affine combinations of $a$ and $b$. 
Since $a$ and $b$ were arbitrary, this completes the proof. 
\end{proof}

At this point we point out that in the proof of Theorem \ref{nSDOPP_thm} we will only need the classical version of the fundamental theorem of affine geometry where the map is bijective and the preservation of collinearity is assumed in both directions.
The reason why we stated Theorem \ref{aff_thm} will be revealed right after Corollary \ref{nn_cor}.
We proceed with the verification of the following lemma.

\begin{lem}\label{inj&affind_lem}
Let $n \geq 2$, $X$ and $Y$ be two $n$-normed spaces, and $\phi\colon X\to Y$ be an $n$DOPP transformation. Then the following conditions are fulfilled:
\begin{itemize}
\item[(i)] for every $2 \leq k \leq n$, $\phi$ preserves affine independence of $k$ vectors in one direction,
\item[(ii)] $\phi$ is injective,
\item[(iii)] $\ran\phi$ is not contained in any affine line.
\end{itemize}
\end{lem}

\begin{proof}
Let $x_0, x_1, \dots x_{n-1} \in X$ be a system of affine independent vectors. 
Since $\dim X \geq n$, there exists a vector $x_n$ such that $\|x_1-x_0,\dots, x_n-x_0\| = 1$, and hence that $\|\phi(x_1)-\phi(x_0),\dots, \phi(x_n)-\phi(x_0)\| = 1$. 
This completes the case $k=n$ in (i).
For the $k<n$ case we simply find some vectors $x_k,\dots x_{n-1}\in X$ such that the system $x_0,\dots x_{k-1},x_k,\dots x_{n-1}$ is still affine independent.

We observe that (ii) simply means (i) in the $k = 2$ case, and that (iii) follows easily from \eqref{nDOPP_eq}.
\end{proof}

We have the following easy consequence of Lemma \ref{inj&affind_lem} and Corollary \ref{aff_cor}.

\begin{cor}\label{nn_cor}
Let $n\geq 2$, and assume that $\phi\colon X\to Y$ is a lineation which satisfies \eqref{nDOPP_eq}. 
Then $\phi$ is an affine $n$-isometry.
\end{cor}

Next, let us observe that Corollary \ref{nn_cor} implies the following results: 
\cite[Theorem 3.1]{2007_nn}, \cite[Theorems 2.10 and 2.12]{2004_nn2}, \cite[Theorems 4 and 6]{2004_nn} and \cite[Theorem 3.6]{2010_nn}. 
Namely, if in \cite[Theorem 3.1]{2007_nn} we relax the assumption about 2-isometriness and instead simply assume that $f$ is a 2DOPP mapping, then we basically get the statement of Corollary \ref{nn_cor} when $n=2$. 
In \cite[Theorem 2.10]{2004_nn2} if we drop the $n$-Lipschitz and the $n$-collinearity assumptions, then again we obtain the statement of Corollary \ref{nn_cor}. We can deal similarly with \cite[Theorem 2.12]{2004_nn2}. 
The paper \cite{2004_nn} considers 2-isometries. Similarly as before, the conditions about 2-Lipschitzness in \cite[Theorem 4]{2004_nn} can be deleted and we still get the same conclusion. Also, in \cite[Theorem 6]{2004_nn} the property (*) is unnecessary to assume.
Finally, let us consider the statement of \cite[Theorem 3.6]{2010_nn}. As it was shown in \cite[Lemma 3.3]{2010_nn}, any map satisfying (i)--(iii) automatically preserves $n$-distance $\rho$. Also (iii) implies that $f$ is a lineation. Therefore, as in Corollary \ref{nn_cor}, we immediately infer that $f$ is an $n$-isometry. So the assumption about weak-$n$-isometriness is excrescent.

We proceed with proving the Mazur--Ulam theorem for $n$-normed spaces. 

\begin{cor}
Every $n$-isometry $\phi\colon X\to Y$ is automatically affine ($n\geq 2$).
\end{cor}

\begin{proof}
Lemma \ref{inj&affind_lem} gives that $\phi$ is injective and its range is not contained in any affine line. 
If $n = 2$, then \eqref{nI_eq} ensures that $\phi$ is a lineation, and by Corollary \ref{aff_cor} we are done.

On the other hand, if $n > 2$, then an easy calculation gives that for every $a\in X$ the map $X \ni x-a \mapsto \phi(x) - \phi(a) \in Y$ is an $n$-norm preserver.
Thus Theorem \ref{pmlin_nn_thm} implies that every affine line going through $a$ is mapped into an affine line which goes through $\phi(a)$.
Therefore $\phi$ is a lineation, which completes the proof.
\end{proof}

The above corollary was obtained e.g.~in \cite[Theorem 3.3]{2009_nn}.
We proceed with proving our result on surjective $n$SDOPP transformations.

\begin{proof}[Proof of Theorem \ref{nSDOPP_thm}]
Let us assume first that $n\geq 3$.
Then by Lemma \ref{inj&affind_lem} $\phi$ is a bijective lineation, moreover, the same holds for the inverse $\phi^{-1}\colon Y \to X$. 
An easy application of the classical version of the fundamental theorem of affine geometry (or also Corollary \ref{aff_cor}) gives that $\phi$ is affine, and therefore it is an $n$-isometry.

Now, assume that $n = 2$.
We obviously have that both $\phi$ and $\phi^{-1}$ are bijective 2SDOPP maps.
Therefore it is enough to show that $\phi$ is a lineation, because then we can use the classical version of the fundamental theorem of affine geometry. 
Let us consider three different collinear points $x_0,x_1,x_2$, and assume that $\phi(x_0),\phi(x_1),\phi(x_2)$ are not collinear. 
We may assume, by re-indexing these three points if necessary, that $x_1 \neq \tfrac{1}{2}(x_0+x_2)$. 
Then by \cite[Theorem 3.1]{2010_nn}, we can find a $\phi(x) \in Y$ such that
\[
\|\phi(x_0)-\phi(x),\phi(x_1)-\phi(x)\| = \|\phi(x_1)-\phi(x),\phi(x_2)-\phi(x)\| = 1.
\]
But this implies 
\[
\|x_0-x_1,x_1-x\| = \|x_0-x,x_1-x\| = 1 = \|x_1-x,x_2-x\| = \|x_1-x,x_2-x_1\|,
\]
and thus, by collinearity and ($n$N3), we get $x_2-x_1\in\{x_0-x_1,x_1-x_0\}$, a contradiction. 
Therefore $\phi$ is indeed a lineation. 
\end{proof}

We have the following Lester--Martin type consequence. 

\begin{cor}\label{nSDOPP_cor}
Let $n\geq 2$, $E$ and $F$ be two real inner product spaces, and $\phi\colon E\to F$ be a surjective transformation which satisfies the following condition:
\begin{equation}\label{nSVOPP_eq}\tag{$n$SVOPP}
\begin{gathered}
\Area_n(x_1-x_0,\dots, x_n-x_0) = 1 \quad \iff \\ \Area_n(\phi(x_1)-\phi(x_0),\dots, \phi(x_n)-\phi(x_0)) = 1 \qquad
(x_0,\dots x_n\in E),
\end{gathered}
\end{equation}
Then $\phi$ is an isometry if $\dim E \geq n+1$, and an equiaffinity if $\dim E = n$.
\end{cor}

\begin{proof}
Theorem \ref{nSDOPP_thm} gives us that $\phi$ is an affine transformation. 
An easy application of Corollary \ref{T:RW_n_cor} completes the proof.
\end{proof}

It would be interesting to explore whether the conclusion of Theorem \ref{pmlin_nn_thm} holds in the $n=2$ case, even for the very special 2-norm $\Area_2(\cdot,\cdot)$.
Note that by \cite[Theorem 1]{G6}, we have this conclusion for Hilbert spaces and bijective transformations. 
A possible way to attack this problem could be to show that we have the conclusion of Theorem \ref{pmlin_nn_thm} in the case when $\dim X = 2 = n$, and then apply the fundamental theorem of projective geometry for the general case.

It would be also interesting to see to what extent the assumptions of Theorem \ref{nSDOPP_thm} can be relaxed.
We suspect that in the most general case, i.e. for (not necessarily onto) $n$DOPP maps, there must be counterexamples, even for $\Area_n(\cdot,\dots,\cdot)$.
However, to the best of our knowledge, no counterexamples have been provided so far.



\begin{thebibliography}{ams}

\bibitem{Al}
A.~D.~Aleksandrov, 
Mappings of families of sets, 
\emph{Soviet Math.~Dokl.} {\bf 11} (1970), 116--120.

\bibitem{BeQu}
F.~S.~Beckman and D.~A.~Quarles, 
On isometries of Euclidean spaces,
\emph{Proc.~Amer.~Math.~Soc.} {\bf 4} (1953), 810--815.

\bibitem{Be} 
W.~Benz,
An elementary proof of the theorem of Beckman and Quarles, 
\emph{Elem.~Math.} {\bf 42} (1987), 4--9.

\bibitem{Be2}
W.~Benz,
\emph{Real Geometries},
Bibliographisches Institut, Mannheim, 1994.

\bibitem{BeBe}
W.~Benz and H.~Berens,
A contribution to a theorem of Ulam and Mazur,
\emph{Aequationes Math.} {\bf 34} (1987), 61--63.

\bibitem{CaVo}
D.~S.~Carter and A.~Vogt,
Collinearity-preserving functions between Desarguesian planes,
\emph{Mem.~Amer.~Math.~Soc.} {\bf 27} (1980), 1--98.

\bibitem{2010_nn}
X.~Y.~Chen and M.~M.~Song,
Characterizations on isometries in linear $n$-normed spaces, 
\emph{Nonlinear Anal.} {\bf 72} (2010), 1895--1901.

\bibitem{2007_nn}
H.~Y.~Chu, 
On the Mazur--Ulam problem in linear 2-normed spaces, 
\emph{J.~Math.~Anal.~Appl.} {\bf 327} (2007) 1041--1045.

\bibitem{2009_nn}
H.~Y.~Chu, S.~K.~Choi and D.~S.~Kang,
Mappings of conservative distances in linear $n$-normed spaces,
\emph{Nonlinear Anal.} {\bf 70} (2009) 1068--1074.

\bibitem{2004_nn2}
H.~Y.~Chu, K.~H.~Lee and C.~K.~Park, 
On the Aleksandrov problem in linear n-normed spaces,
\emph{Nonlinear Anal.} {\bf 59} (2004) 1001--1011.

\bibitem{2004_nn}
H.~Y.~Chu, C.~G.~Park and W.~G.~Park, 
The Aleksandrov problem in linear 2-normed spaces, 
\emph{J.~Math.~Anal.~Appl.} {\bf 289} (2004) 666--672.

\bibitem{De} B.~V.~Dekster,
Nonisometric distance 1 preserving mapping $E^2\to E^6$,
\emph{Arch. Math. (Basel)} {\bf 45} (1985), 282--283.

\bibitem{DiGaWh}
C.~Diminnie, S.~G\"ahler and A.~White,
Strictly convex linear 2-normed spaces,
\emph{Math.~Nachr.} {\bf 59} (1974), 319--324.

\bibitem{Fa}
C.-A.~Faure,
An elementary proof of the fundamental theorem of projective geometry,
\emph{Geometriae Dedicata} {\bf 90} (2002), 145--151.

\bibitem{Ga}
S.~G\"ahler,
Lineare 2-normierte R\"aume,
\emph{Math.~Nachr.} {\bf 28} (1964), 1--43.

\bibitem{Ge} 
Gy.~P.~Geh\'er, 
A contribution to the Aleksandrov conservative distance problem in two dimensions,
\emph{Linear Algebra Appl.}, {\bf 481} (2015), 280--287.

\bibitem{G6} 
Gy.~P.~Geh\'er, 
Maps on real Hilbert spaces preserving the area of parallelograms and a preserver problem on self-adjoint operators, 
\emph{J.~Math.~Anal.~Appl.}
{\bf 422} (2015), 1402--1413.

\bibitem{HuSe}
W.-L.~Huang and P.~\v Semrl,
The optimal version of Hua's fundamental theorem of geometry of square matrices -- the low dimensional case,
\emph{Linear Algebra Appl.}, to appear. DOI:10.1016/j.laa.2014.08.014

\bibitem{Ju}
R.~Juh\'asz,
Another proof of the Beckman--Quarles theorem,
\emph{Adv.~Geom.}, \textbf{15} (2015), 519--521.

\bibitem{Le2}
J.~A.~Lester,
Euclidean plane point-transformations
preserving unit area or unit perimeter,
\emph{Arch.~Math.}, {\bf 45} (1985), 561--564.

\bibitem{Le}
J.~A.~Lester,
Martin's theorem for Euclidean n-space and a generalization to the perimeter case,
\emph{J.~Geom.} {\bf 27} (1986), 29--35.

\bibitem{Mi}
A.~Misiak, 
$n$-inner product spaces, 
\emph{Math.~Nachr.} {\bf 140} (1989), 299--319.

\bibitem{RaSe}
T.~M.~Rassias and P.~\v Semrl,
On the Mazur--Ulam theorem and the Aleksandrov problem for unit distance preserving mappings, 
\emph{Proc.~Amer.~Math.~Soc.} {\bf 118} (1993), 919--925.

\bibitem{Sch}
H.~Schaeffer, 
\"Uber eine Verallgemeinerung des Fundamentalsatzes in desarguesschen affinen Ebenen, 
\emph{Techn.~Univ.~M\"unchen} TUM-M8010.

\bibitem{Tyszka}
A.~Tyszka,
A discrete form of the Beckman-Quarles theorem for two-dimensional strictly convex normed spaces,
\emph{Nonlinear Funct. Anal. Appl.} \textbf{7} (2002), 353--360.

\end{thebibliography}
\end{document}